\documentclass[11pt]{article}

\usepackage[margin=1.08in]{geometry}
\usepackage{amsmath,amssymb,amsthm,mathtools}
\usepackage{microtype}
\usepackage{enumitem}
\usepackage{booktabs}
\usepackage{hyperref}
\usepackage[nameinlink,capitalize]{cleveref}
\usepackage{xcolor}

\hypersetup{
  colorlinks=true,
  linkcolor=blue!55!black,
  citecolor=blue!55!black,
  urlcolor=blue!55!black,
  pdftitle={A proof of the fractional Lane--Emden conjecture},
  pdfauthor={Linfeng Mei and Juncheng Wei}
}
\numberwithin{equation}{section}
\newtheorem{theorem}{Theorem}[section]

\newtheorem{proposition}[theorem]{Proposition}
\newtheorem{lemma}[theorem]{Lemma}
\newtheorem{corollary}[theorem]{Corollary}
\theoremstyle{definition}

\theoremstyle{remark}\newtheorem{remark}[theorem]{Remark}

\newcommand{\R}{\mathbb{R}}
\newcommand{\Sph}{\mathbb{S}}
\newcommand{\E}{\mathbb{E}}
\newcommand{\Pp}{\mathbb{P}}
\newcommand{\dd}{\,\mathrm{d}}

\newcommand{\Id}{\mathrm{Id}}
\newcommand{\eps}{\varepsilon}

\title{A proof of the fractional Lane--Emden conjecture}
\author{Linfeng Mei\thanks{Zhejiang Normal University. E-mail: \href{mailto:lfmei@zjnu.edu.cn}{lfmei@zjnu.edu.cn}.} \and Juncheng Wei\thanks{The Chinese University of Hong Kong. E-mail: \href{mailto:wei@math.cuhk.edu.hk}{wei@math.cuhk.edu.hk}.}}
\date{October 7, 2026}

\begin{document}
\maketitle

\begin{abstract}
Let $n\ge2$, $0<s<1$, and $p,q>0$. We prove a Liouville theorem for the fractional Lane--Emden system
\[
 (-\Delta)^s u=v^p,\qquad (-\Delta)^s v=u^q, \qquad u, v>0 
 \qquad\text{in }\R^n.
\]
More precisely, we show that the system does not have a strictly positive entire solution whenever
\[
 \frac1{p+1}+\frac1{q+1}>\frac{n-2s}{n}.
\]
No radial symmetry, global boundedness, finite-energy assumption, or prescribed decay at infinity is required. Our proof adapts the localized potential-kernel virial and source-layer interval-pressure method recently introduced for the classical H\'enon--Lane--Emden system. The basic reason the method also works in the fractional setting is simple: the Riesz kernel associated with $(-\Delta)^s$ has the same logarithmic derivative structure as the Newton kernel, with $n-2$ replaced by $n-2s$. However, two genuinely nonlocal points require special treatments. When $2s\le1$, the individual kernel-gradient terms are not locally absolutely integrable, so the two equations must be combined before the diagonal limit is taken. When $2s>n-1$, the kernel restricted to a line is not integrable at infinity, and the interval comparison must therefore be formulated using nonnegative extended integrals. After these issues are resolved, a localized virial identity, a one-dimensional interval-pair inequality, and an averaged source-layer estimate yield a universal local energy bound. The subcritical scaling then forces the same energy to diverge, giving the desired contradiction.
\end{abstract}



\section{Introduction}
\setcounter{equation}{0}
For $0<s<1$, let
\[
 K_s(x)=c_{n,s}|x|^{2s-n},\qquad x\in\R^n\setminus\{0\},
\]
where $c_{n,s}>0$ is chosen so that $(-\Delta)^sK_s=\delta_0$ in the sense of distributions \cite{Landkof1972,SteinWeiss1958}. (The precise normalization does not matter.) We begin with the unweighted fractional Lane--Emden system in its integral form
\begin{equation}\label{eq:system}
 \begin{cases}
 u=K_s*v^p,\\
 v=K_s*u^q,
 \end{cases}
 \qquad u,v>0\quad\text{in }\R^n.
\end{equation}
The critical hyperbola associated with \eqref{eq:system} is
\begin{equation}\label{eq:critical-hyperbola}
 \frac1{p+1}+\frac1{q+1}=\frac{n-2s}{n}.
\end{equation}
This integral system is closely tied to the Hardy--Littlewood--Sobolev inequality \cite{DouZhu2015,Lieb1983,SteinWeiss1958,Villavert2015}. In the standard whole-space setting it is also equivalent to the fractional PDE system
\[
 (-\Delta)^s u=v^p,\qquad (-\Delta)^s v=u^q \qquad u, v>0 
 \qquad\text{in }\R^n.
\]
Since we shall work mainly with the potential formulation, we first make the short argument which shows equivalence between the two forms. We use the usual whole-space class
\[
 \mathcal L_s(\R^n)
 :=\left\{h\in L^1_{\rm loc}(\R^n):
 \int_{\R^n}\frac{|h(x)|}{1+|x|^{n+2s}}\,\dd x<\infty\right\},
\]
and the standard pointwise definition \cite{DiNezza2012}
\[
 (-\Delta)^s h(x)
 =C_{n,s}\,\operatorname{P.V.}\int_{\R^n}
 \frac{h(x)-h(y)}{|x-y|^{n+2s}}\,\dd y.
\]

\begin{proposition}[PDE and potential forms]\label{prop:pde-potential}
Let $u,v>0$ belong to $C^{1,1}_{\rm loc}(\R^n)\cap\mathcal L_s(\R^n)$ and satisfy
\[
 (-\Delta)^s u=v^p,\qquad (-\Delta)^s v=u^q
 \qquad\text{in }\R^n.
\]
Then
\begin{equation}\label{eq:pde-to-potential}
 u=K_s*v^p,\qquad v=K_s*u^q
 \qquad\text{in }\R^n.
\end{equation}
Conversely, every finite positive solution of \eqref{eq:pde-to-potential} solves the fractional system in the distributional sense; after the local regularity established below, the equations hold in the usual pointwise sense whenever the pointwise definition of the fractional Laplacian is used.
\end{proposition}

\begin{proof}
The argument is standard, but we include it for completeness. Let $G_R(x,y)$ denote the Green kernel of $(-\Delta)^s$ in $B_R$ with zero exterior data, and set
\[
 U_R(x)=\int_{B_R}G_R(x,y)v(y)^p\,\dd y,
 \qquad x\in B_R,
\]
extending $U_R$ by zero outside $B_R$.  Then
\[
 (-\Delta)^sU_R=v^p\quad\text{in }B_R,
 \qquad U_R=0\quad\text{on }B_R^c.
\]
Thus $u-U_R$ is $s$-harmonic in $B_R$ and agrees with the positive function $u$ on $B_R^c$. The maximum principle therefore gives
\[
 0\le U_R(x)\le u(x).
\]
As $R\to\infty$, domain monotonicity of the Green kernels and the explicit ball Green function (see, e.g., \cite{ChenSong1998,RosOton2016}) give
\[
 G_R(x,y)\uparrow K_s(x-y)
 \qquad (x\ne y).
\]
Monotone convergence now yields
\[
 I_{2s}(v^p)(x):=\int_{\R^n}K_s(x-y)v(y)^p\,\dd y\le u(x),
\]
so this potential is finite.  Set
\[
 h_u=u-I_{2s}(v^p)\ge0.
\]
Since $K_s$ is the fundamental solution of $(-\Delta)^s$,
\[
 (-\Delta)^s h_u=0
\]
in all of $\R^n$ in the distributional sense.  Also $0\le h_u\le u$, hence $h_u\in\mathcal L_s(\R^n)$.  The Liouville theorem for nonnegative entire $s$-harmonic functions implies that $h_u$ is constant; see, for example, \cite{ChenLiOu2006,JarohsWeth2016,ZhuoChenCuiYuan2014}.  Thus
\[
 u=c_u+I_{2s}(v^p),\qquad c_u\ge0.
\]
The same argument gives
\[
 v=c_v+I_{2s}(u^q),\qquad c_v\ge0.
\]
In fact, both constants vanish. If $c_u>0$, then $u\ge c_u$, and hence
\[
 v(x)\ge I_{2s}(u^q)(x)
 \ge c\,c_u^q\int_{|y|>R}|x-y|^{2s-n}\,\dd y=\infty,
\]
because the radial integral at infinity is a positive multiple of
\[
 \int_R^\infty r^{2s-1}\,\dd r,
\]
which diverges for every $s>0$. This contradicts the finiteness of $v$. Thus $c_u=0$; the same argument gives $c_v=0$, and \eqref{eq:pde-to-potential} follows.

Conversely, applying $(-\Delta)^s$ distributionally to $K_s*F$ gives $F$, since $(-\Delta)^sK_s=\delta_0$.  Applying this with $F=v^p$ and $F=u^q$ proves the reverse statement.
\end{proof}

The curve above is exactly the critical curve predicted by scaling. A large literature on HLS-type systems, integral equations, and fractional Lane--Emden problems provides classification and Liouville theorems under additional assumptions or in restricted exponent ranges; see, for example, \cite{ChenLiOu2005,ChenLiOu2006,ChengHuangLi2016,DAmbrosioMitidieri2014,Hang2007,Le2021,Lieb1983}. Our purpose here is to remove those auxiliary assumptions in the unweighted, equal-order subcritical problem: we impose no radial symmetry, no global boundedness, no finite-energy condition, and no decay hypothesis at infinity.

In the local case $s=1$ this is the classical Lane--Emden system, whose subcritical nonexistence theory is the Lane--Emden conjecture; see \cite{ChenLi2009,ChengHuang2019,GidasSpruck1981a,GidasSpruck1981b,Mitidieri1993,Mitidieri1996,PolacikQuittnerSouplet2007,SerrinZou1996,Souplet2009} for the unweighted problem, and \cite{BidautVeron1989,BidautVeronGiacomini2010,FazlyGhoussoub2014,HuangZou2026,KuiLiLiWei2025,Li2022,LiZhang2019,MitidieriPohozaev2001,Phan2012} for the weighted H\'enon--Lane--Emden counterpart, which contains the unweighted case; the corresponding subcritical theorem with arbitrary real weights was established recently in \cite{OpenAI2026}. Related half-space and conical Liouville theorems for the local problem appear in \cite{DupaigneFarinaPetitt2023,DupaigneSirakovSouplet2022,LiSouplet2025}. For the fractional Laplacian, the available Liouville-type results either hold in restricted exponent ranges \cite{QuaasXia2016}, restrict the geometry or the domain \cite{ChenFangYang2015,ChenHuyuan2018,DaiLiuLu2016,DaiQin2021,FallWeth2016,QuaasXia2014}, or require additional hypotheses such as stability or finite Morse index \cite{DavilaDupaigneWei2017,DuongNguyen2021,FazlyWei2017,Hajlaoui2022}.

We state the main result first for the exact Riesz-potential system. Proposition~\ref{prop:pde-potential} then gives the corresponding nonexistence theorem for positive classical solutions of the fractional Lane--Emden PDE in $C^{1,1}_{\rm loc}(\R^n)\cap\mathcal L_s(\R^n)$.

\begin{theorem}[Subcritical fractional Lane--Emden Liouville theorem]\label{thm:main}
Let $n\ge2$, $0<s<1$, and $p,q>0$. There do not exist finite continuous functions $u,v:\R^n\to(0,\infty)$ satisfying \eqref{eq:system} pointwise if
\begin{equation}\label{eq:subcritical}
 \frac1{p+1}+\frac1{q+1}>\frac{n-2s}{n}.
\end{equation}
\end{theorem}

\begin{corollary}[Fractional PDE form]\label{cor:pde}
Let $n\ge2$, $0<s<1$, and $p,q>0$ satisfy \eqref{eq:subcritical}.  Then the fractional Lane--Emden system
\[
 (-\Delta)^s u=v^p,\qquad (-\Delta)^s v=u^q
 \qquad\text{in }\R^n
\]
has no strictly positive solution $(u,v)$ with
\[
 u,v\in C^{1,1}_{\rm loc}(\R^n)\cap\mathcal L_s(\R^n).
\]
\end{corollary}

\begin{proof}
Proposition~\ref{prop:pde-potential} reduces the PDE system to \eqref{eq:system}, and Theorem~\ref{thm:main} applies.
\end{proof}

\paragraph{Idea of the proof.}
The main point is to treat the two equations as one coupled interaction rather than estimate $u$ and $v$ separately. We encode that interaction in a positive measure on pairs of points $(x,y)$. A localized virial identity then describes how the interaction responds to a radial dilation. Its troublesome part is a directional ``pressure'' term supported near the edge of the cutoff. To estimate it, we slice the superlevel sets of $u^q$ and $v^p$ along directions that are almost radial. On each pair of one-dimensional source intervals, a simple integration-by-parts inequality converts interior interaction into endpoint interaction plus a nonnegative directional term. Because the endpoints lie on fixed source levels, their values are explicit. After averaging over directions, the remaining errors are either a small multiple of the local energy or a universal additive constant. This produces a uniform local energy bound, while subcritical scaling forces that same energy to grow without bound.

The proof builds on the localized potential-kernel virial and source-layer interval-pressure method introduced in the recent proof of the subcritical H\'enon--Lane--Emden conjecture by OpenAI \cite{OpenAI2026}. The fractional extension has its own difficulties. Two nonlocal features need to be handled separately. First, when $2s\le1$, the individual kernel-gradient terms fail to be locally absolutely integrable. We therefore cut out the diagonal, combine the two equations, and only then let the truncation vanish; the factor $X(x)-X(y)$ restores exactly one power of $|x-y|$. Second, when $2s>n-1$, the restriction of $K_s$ to a line is not integrable at infinity. For that reason the interval-pair comparison is kept in a nonnegative extended-integral form until every term that will later be subtracted is known to be finite.

Here is the structure of the argument. Set
\[
 f=u^q,\qquad g=v^p,
\]
and introduce the positive interaction measure
\[
 \dd\pi(x,y)=f(x)g(y)K_s(x-y)\,\dd x\,\dd y.
\]
Its marginals are $fu\,\dd x$ and $gv\,\dd y$. A compactly supported vector field $X=x\eta$ yields a localized fractional Riesz virial identity with coefficient $n-2s$, in the spirit of the Pohozaev-type identities for the fractional Laplacian \cite{FallWeth2012,RosOtonSerra2014a,RosOtonSerra2015}. A high-order radial cutoff localizes the virial kernel to
\[
 \eta(x)-w(x)D(x,y),
 \qquad
 w=-x\cdot\nabla\eta,
 \qquad
 D(x,y)=\frac{((x-y)\cdot\omega_x)^2}{|x-y|^2}.
\]
The term involving $wD$ is the pressure. To control it, we average over directions $e$ close to $\omega_x$, decompose the superlevel sets of $f$ and $g$ into one-dimensional intervals, and apply a weighted interval-pair inequality to the restricted fractional Riesz kernel. The endpoint terms recover the original Riesz potentials, while the exact layer identity
\[
 \int_0^{u(x)^q}\bigl(u(x)-t^{1/q}\bigr)\,\dd t
 =\frac{u(x)^{q+1}}{q+1}
\]
recovers the nonlinear energy. Short intervals see almost no variation of the cap weights, long intervals at high source levels occupy only a small set of directions, and low levels contribute only a universal additive error. These three observations give the desired local energy bound. The strict subcritical inequality is exactly what makes the corresponding scaling exponent positive, and dilation then produces the contradiction.

\paragraph{Relation with the Newtonian interval-pressure method.}
The geometric reason for the argument is the logarithmic derivative identity
\begin{equation}\label{eq:logder-intro}
 \nabla K_s(z)=-(n-2s)\frac{z}{|z|^2}K_s(z).
\end{equation}
The Newton kernel is simply the Riesz kernel of order $2$. In the present problem that order becomes $2s\in(0,2)$. The calculations below show that the interval-pressure mechanism is driven by the homogeneity and positivity of the kernel, not by locality of the Laplacian itself.

\paragraph{Organization.}
Section \ref{sec:prelim} gives scaling, source-mass estimates, truncated potentials, and local regularity. Section \ref{sec:virial} proves the fractional Riesz virial identity and the localization lemma. Section \ref{sec:interval} establishes the weighted interval-pair inequality. Section \ref{sec:pressure} proves the localized source-layer pressure estimate. Section \ref{sec:closure} derives the universal local energy bound and the scaling contradiction.

\section{Scaling and universal potential estimates}\label{sec:prelim}

Throughout the paper, $C$ denotes a positive constant that may depend on $n,s,p,q$ and on any auxiliary parameters displayed explicitly, but never on the particular solution under consideration. Set
\begin{equation}\label{eq:f-g-def}
 f=u^q,\qquad g=v^p.
\end{equation}
Then \eqref{eq:system} reads
\begin{equation}\label{eq:potential}
 u=K_s*g,\qquad v=K_s*f.
\end{equation}

\subsection{The small-product exclusion and scaling}

\begin{lemma}\label{lem:pq}
Every positive finite continuous solution of \eqref{eq:system} satisfies $pq>1$.
\end{lemma}

\begin{proof}
For $R>0$, set
\[
 m_u(R)=\inf_{B_R}u,\qquad m_v(R)=\inf_{B_R}v.
\]
Both quantities are positive. If $x,y\in B_R$, then $|x-y|\le2R$ and hence $K_s(x-y)\ge cR^{2s-n}$. It follows that
\[
 m_u(R)\ge cR^{2s} m_v(R)^p,
 \qquad
 m_v(R)\ge cR^{2s} m_u(R)^q.
\]
Combining the inequalities gives
\[
 m_u(R)^{1-pq}\ge cR^{2s(1+p)}.
\]
If $pq<1$, the left-hand side is at most $u(0)^{1-pq}$, while the right-hand side tends to infinity. If $pq=1$, the left-hand side is one and the same contradiction follows.
\end{proof}

From now on we may assume $pq>1$. Define
\begin{equation}\label{eq:alpha-beta}
 \alpha=\frac{2s(p+1)}{pq-1},
 \qquad
 \beta=\frac{2s(q+1)}{pq-1}.
\end{equation}
Then
\begin{equation}\label{eq:scaling-rel}
 \alpha+2s=p\beta,
 \qquad
 \beta+2s=q\alpha.
\end{equation}
For $R>0$, set
\begin{equation}\label{eq:scaling}
 u_R(x)=R^\alpha u(Rx),
 \qquad
 v_R(x)=R^\beta v(Rx).
\end{equation}
A change of variables shows that $(u_R,v_R)$ again solves \eqref{eq:system} with the same kernel normalization.

\subsection{Source masses}

\begin{proposition}[Universal centered source estimates]\label{prop:source}
There exists $C$ such that every solution satisfies, for all $R>0$,
\begin{equation}\label{eq:source-bounds}
 \int_{B_R}f\le CR^{n-2s-\beta},
 \qquad
 \int_{B_R}g\le CR^{n-2s-\alpha}.
\end{equation}
Moreover,
\begin{equation}\label{eq:L1-local}
 \int_{B_4}(f+g+u+v)\le C,
\end{equation}
and, for $x\in B_2$ and $0<\ell\le1$,
\begin{equation}\label{eq:tail}
 \int_{|x-y|\ge\ell}K_s(x-y)(f(y)+g(y))\,\dd y
 \le C\ell^{2s-n}.
\end{equation}
\end{proposition}

\begin{proof}
Let
\[
 M_f=\int_{B_1}f,\qquad M_g=\int_{B_1}g.
\]
For $x\in B_1$, the kernel has a uniform positive lower bound on $B_1\times B_1$, hence
\[
 u(x)\ge cM_g,\qquad v(x)\ge cM_f.
\]
It follows that
\[
 M_f\ge cM_g^q,\qquad M_g\ge cM_f^p,
\]
and therefore $M_f\ge cM_f^{pq}$. By \cref{lem:pq}, $M_f+M_g\le C$. Applying this unit-ball estimate to the scaled pair \eqref{eq:scaling} and changing variables yields \eqref{eq:source-bounds}.

For $x\in B_4$, decompose the contribution to the potential from $\R^n\setminus B_8$ into annuli $2^j\le|y|<2^{j+1}$, $j\ge3$. By \eqref{eq:source-bounds}, the $g$-contribution is bounded by
\[
 C\sum_{j\ge3}2^{j(2s-n)}\int_{B_{2^{j+1}}}g
 \le C\sum_{j\ge3}2^{-j\alpha}<\infty,
\]
and similarly for $f$. On the near region, Tonelli's theorem and
\[
 \sup_{y\in B_8}\int_{B_4}K_s(x-y)\,\dd x<\infty
\]
combine with the source-mass bounds to give \eqref{eq:L1-local}.

For the tail estimate, observe that $K_s(x-y)\le C\ell^{2s-n}$ whenever $|x-y|\ge\ell$. The part over $B_8$ is therefore bounded by $C\ell^{2s-n}$, while the annular estimate above bounds the contribution from $\R^n\setminus B_8$ by $C$. Since $\ell^{2s-n}\ge1$ for $0<\ell\le1$, this proves \eqref{eq:tail}.
\end{proof}

\subsection{Local regularity}

We shall need only a standard local smoothing property of Riesz potentials. For background one may consult Adams--Hedberg \cite{AdamsHedberg1996}; related regularity lifting arguments for fractional Lane--Emden systems appear in \cite{ChenJinLiLim2005}; potential-theoretic representation formulae for higher-order and fractional problems may be found in \cite{CaristiDAmbrosioMitidieri2008}, and we refer to \cite{CaffarelliSilvestre2007,RosOtonSerra2014b,Silvestre2007} for the regularity theory of the fractional Dirichlet problem.

\begin{lemma}[Local smoothing]\label{lem:regularity}
Let $(u,v)$ be a finite continuous positive solution of \eqref{eq:system}. Then $u,v,f,g$ are $C^1$ on every compact subset of $\R^n$. More precisely, if $2s\le1$, then after finitely many local bootstrap steps one may arrange
\[
 f,g\in C^{0,\theta}_{\mathrm{loc}}(\R^n)
 \quad\text{for some }\theta>1-2s.
\]
\end{lemma}

\begin{proof}
Fix $K\Subset K_1$ and choose $\chi\in C_c^\infty(K_1)$ equal to one near $K$. Decompose
\[
 u=K_s*(\chi g)+K_s*((1-\chi)g).
\]
The second term is smooth on $K$: derivatives of the kernel are uniformly dominated there by a fixed multiple of $K_s(x_0-y)$ for some $x_0\in K$, which is integrable against $g$ by \eqref{eq:potential}. For the compactly supported near part, the standard local fractional Riesz-potential estimate gives a gain of any H\"older order strictly smaller than $2s$ (with the usual Zygmund interpretation at integer thresholds). Since $u,v$ are strictly positive on $K_1$, the maps $t\mapsto t^p$ and $t\mapsto t^q$ are smooth on the ranges of $u,v$ there. Thus each regularity gain for $u,v$ is inherited by $f,g$ and can be iterated. A finite number of iterations gives $C^1$ regularity, and, when $2s\le1$, further iteration gives a H\"older exponent $\theta>1-2s$ for $f,g$.
\end{proof}

\section{A localized fractional Riesz virial identity}\label{sec:virial}

\subsection{The interaction measure and the diagonal cancellation}

Define the nonnegative interaction measure
\begin{equation}\label{eq:pi}
 \dd\pi(x,y)=f(x)g(y)K_s(x-y)\,\dd x\,\dd y.
\end{equation}
By Tonelli and \eqref{eq:potential},
\begin{equation}\label{eq:marginals}
 \pi(\dd x,\R^n)=f(x)u(x)\,\dd x,
 \qquad
 \pi(\R^n,\dd y)=g(y)v(y)\,\dd y.
\end{equation}

\begin{lemma}[Localized fractional Riesz virial identity]\label{lem:virial-general}
Let $X\in C_c^1(\R^n;\R^n)$ and set
\begin{equation}\label{eq:PhiX}
 \Phi_X(x,y)=\frac{(X(x)-X(y))\cdot(x-y)}{|x-y|^2}.
\end{equation}
Then $\Phi_X\in L^1(\pi)$ and
\begin{equation}\label{eq:virial-general}
 -\int fX\cdot\nabla u
 -\int gX\cdot\nabla v
 =(n-2s)\iint\Phi_X(x,y)\,\dd\pi(x,y).
\end{equation}
\end{lemma}

\begin{proof}
Because $X$ is Lipschitz,
\[
 |\Phi_X(x,y)|\le \|\nabla X\|_\infty
\]
when $|x-y|\le1$. Hence the diagonal singularity in $\Phi_X\,\dd\pi$ is no worse than $|x-y|^{2s-n}$, which is locally integrable for every $s>0$. Away from the diagonal, at least one endpoint lies in a fixed compact set whenever $\Phi_X\ne0$, and the marginal identities \eqref{eq:marginals} give absolute integrability.

For $\eps>0$, define
\[
 G_\eps^g(x)=\int_{|x-y|>\eps}\nabla K_s(x-y)g(y)\,\dd y,
\]
and similarly $G_\eps^f$. If $2s>1$, the gradient kernel is locally integrable and $G_\eps^g\to\nabla u$ locally uniformly. If $0<2s\le1$, use \cref{lem:regularity} with $g\in C^{0,\theta}$ and $\theta>1-2s$. Since $\nabla K_s$ is odd,
\[
 \int_{\eps<|z|<r}\nabla K_s(z)\,\dd z=0,
\]
so, uniformly on a fixed compact set,
\begin{align*}
 |G_{\eps'}^g(x)-G_\eps^g(x)|
 &\le C\int_{\eps'}^\eps r^{2s-n-1}r^\theta r^{n-1}\,\dd r\\
 &\le C\eps^{2s+\theta-1}.
\end{align*}
Thus $G_\eps^g\to\nabla u$ locally uniformly; the same holds for $f,v$.

Set
\[
 I_\eps=-\int fX\cdot G_\eps^g-\int gX\cdot G_\eps^f.
\]
For fixed $\eps$, Fubini is legitimate. Using
\begin{equation}\label{eq:grad-kernel}
 \nabla K_s(z)=-(n-2s)\frac{z}{|z|^2}K_s(z)
\end{equation}
and exchanging variables in the second term gives
\[
 I_\eps=(n-2s)
 \iint_{|x-y|>\eps}\Phi_X(x,y)\,\dd\pi(x,y).
\]
The left side converges to the left side of \eqref{eq:virial-general} by the local uniform convergence of the truncated gradients. The right side converges by the absolute integrability already established.
\end{proof}

\subsection{The cutoff and localized identity}

Set $t_+=\max\{t,0\}$ and define
\begin{equation}\label{eq:cutoff-parameters}
 d(x)=\min\{1,(2-|x|)_+\},
 \qquad
 m=n+3,
 \qquad
 k=n(m+1),
 \qquad
 N=k+n+2.
\end{equation}
We take the radial cutoff
\begin{equation}\label{eq:eta}
 \eta(x)=\bigl[(1-(|x|-1)_+^3)_+\bigr]^N,
 \qquad
 w(x)=-x\cdot\nabla\eta(x),
 \qquad
 X(x)=x\eta(x).
\end{equation}
Then $\eta\in C_c^2(\R^n)$, $\eta=1$ on $B_1$, $w\ge0$, and
\begin{equation}\label{eq:cutoff-estimates}
 d^N\le\eta\le Cd^N,
 \quad
 |\nabla\eta|+w\le Cd^{N-1},
 \quad
 |\nabla^2\eta|+|\nabla w|+|\nabla^2X|\le Cd^{N-2},
 \quad
 dw\le C\eta.
\end{equation}
The high order of vanishing at $|x|=2$ is deliberate; it is the same device used in \cite{OpenAI2026} to make all far-pair errors integrable with room to spare.

Set
\begin{equation}\label{eq:energies}
 E_f=\int\eta fu,
 \quad
 E_g=\int\eta gv,
 \quad
 E=E_f+E_g,
 \quad
 W_f=\int wfu,
 \quad
 W_g=\int wgv,
\end{equation}
and
\begin{equation}\label{eq:a-b}
 a=\frac1{q+1},\qquad b=\frac1{p+1}.
\end{equation}
For $x\ne0$, also set
\begin{equation}\label{eq:D}
 \omega_x=\frac{x}{|x|},
 \qquad
 D(x,y)=\frac{((x-y)\cdot\omega_x)^2}{|x-y|^2}.
\end{equation}
The value of $\omega_0$ is irrelevant.

\begin{proposition}[Localized virial identity]\label{prop:localized-virial}
With
\[
 \Phi(x,y)=\frac{(X(x)-X(y))\cdot(x-y)}{|x-y|^2},
\]
one has
\begin{equation}\label{eq:localized-virial}
 anE_f+bnE_g-aW_f-bW_g
 =(n-2s)\iint\Phi(x,y)\,\dd\pi(x,y).
\end{equation}
\end{proposition}

\begin{proof}
By \cref{lem:virial-general}, it suffices to integrate the left side by parts. Since $f=u^q$,
\[
 -\int fX\cdot\nabla u
 =-\frac1{q+1}\int X\cdot\nabla(u^{q+1})
 =a\int(\operatorname{div}X)u^{q+1}.
\]
Since $\operatorname{div}X=n\eta-w$, this is $anE_f-aW_f$. The second term is identical.
\end{proof}

\subsection{Localization of the virial kernel}

\begin{lemma}[Localization]\label{lem:localization}
For every $\eps>0$, there is $C_\eps$ such that
\begin{equation}\label{eq:localization}
 |E_f-E_g|
 +\iint|\Phi(x,y)-\eta(x)+w(x)D(x,y)|\,\dd\pi(x,y)
 \le \eps E+C_\eps.
\end{equation}
\end{lemma}

\begin{proof}
Fix $0<\delta<1/4$ and write
\[
 \ell(z)=\delta d(z)^m
\]
when $d(z)>0$. Call $(x,y)$ close if $|x-y|\le\ell(z)$ for at least one endpoint $z\in\{x,y\}$ with $d(z)>0$, and far otherwise. Since $d$ is one-Lipschitz, $d$ is comparable along the segment $[x,y]$ for a close pair.

The derivative matrix is
\[
 \nabla X(x)=\eta(x)\Id-w(x)\omega_x\otimes\omega_x.
\]
Taylor's formula and \eqref{eq:cutoff-estimates} give, for a close pair,
\begin{align}
 |\Phi-\eta(x)+w(x)D(x,y)|
 &\le C\delta d(z)^{N-2+m}\le C\delta d(z)^N,\label{eq:close1}\\
 |\eta(x)-\eta(y)|
 &\le C\delta d(z)^{N-1+m}\le C\delta d(z)^N.\label{eq:close2}
\end{align}
Since $d(z)^N\lesssim\eta(x)+\eta(y)$ on close pairs, the marginal identities give a contribution at most $C\delta E$.

For a far pair, every endpoint $z\in B_2$ satisfies $|x-y|>\ell(z)$. Using $0\le D\le1$, $|X(z)|\lesssim d(z)^N$, and \eqref{eq:cutoff-estimates},
\[
 |\Phi-\eta(x)+w(x)D|+|\eta(x)-\eta(y)|
 \le C\sum_{z\in\{x,y\}:d(z)>0}
 \left(d(z)^{N-1}+\frac{d(z)^N}{\ell(z)}\right).
\]
For the terms associated with $x$, \eqref{eq:tail} bounds the integral by
\[
 C\int_{B_2}f(x)
 \left(d(x)^{N-1}+\frac{d(x)^N}{\ell(x)}\right)
 \ell(x)^{2s-n}\,\dd x.
\]
The two exponents of $d$ are
\begin{align*}
 N-1-m(n-2s)&=2n+1+2s(n+3)>0,\\
 N-m(n+1-2s)&=n-1+2s(n+3)>0.
\end{align*}
Thus the far contribution is bounded by a constant $C_{\delta}$. The $y$-terms are identical.

The marginal identities also give
\[
 E_f-E_g=\iint(\eta(x)-\eta(y))\,\dd\pi(x,y).
\]
Combining the close- and far-pair estimates, the left-hand side of \eqref{eq:localization} is bounded by $C\delta E+C_\delta$. Choosing $\delta$ so that $C\delta\le\eps$ proves the lemma.
\end{proof}

\section{A weighted interval-pair inequality}\label{sec:interval}

Fix a unit vector $e\in\Sph^{n-1}$ and two distinct parallel lines. Write
\[
 x=z+se,\qquad y=z'+s'e,
 \qquad z,z'\in e^\perp,
\]
and set
\[
 \rho=|z-z'|>0,\qquad h=s-s'.
\]
The restricted fractional Riesz kernel and its directional factor are
\begin{equation}\label{eq:restricted-kernel}
 \mathcal K(s,s')=c_{n,s}(\rho^2+h^2)^{-(n-2s)/2},
 \qquad
 \mathcal D(s,s')=\frac{h^2}{\rho^2+h^2}.
\end{equation}
For a bounded interval $I=(s_-,s_+)$, define
\[
 \mathcal K_I(s')=\frac{\mathcal K(s_-,s')+\mathcal K(s_+,s')}{2},
\]
and define $\mathcal K_J(s)$ similarly for a bounded interval $J$.

\begin{lemma}[Weighted interval comparison]\label{lem:interval}
Let $I,J\subset\R$ be nonempty open intervals, possibly unbounded. Let $c_I,c_J\ge0$, with $I$ bounded whenever $c_I>0$ and $J$ bounded whenever $c_J>0$. Then
\begin{align}
 &\int_{I\times J}(c_I+c_J)\mathcal K\,\dd s\,\dd s'\nonumber\\
 &\qquad\le
 \int_{I\times J}\Bigl[
 c_I\mathcal K_I+c_J\mathcal K_J
 +\bigl((n-2s)\min(c_I,c_J)\mathcal D+|c_I-c_J|\bigr)\mathcal K
 \Bigr]\,\dd s\,\dd s'.\label{eq:interval-ineq}
\end{align}
All terms are nonnegative extended integrals; an endpoint term with zero coefficient is defined to be zero without evaluating its endpoint average.
\end{lemma}

\begin{proof}
If at most one coefficient is positive, then $|c_I-c_J|=c_I+c_J$ and the result is pointwise. Suppose both are positive, so $I,J$ are bounded. Write $m_I,m_J$ for the midpoints and $d_0=m_I-m_J$. Integration by parts in the first variable gives
\[
 \int_I(\mathcal K-\mathcal K_I)\,\dd s
 =-\int_I(s-m_I)\partial_s\mathcal K\,\dd s.
\]
Since
\[
 \partial_s\mathcal K
 =-(n-2s)\frac{h}{\rho^2+h^2}\mathcal K,
 \qquad
 \partial_{s'}\mathcal K
 =(n-2s)\frac{h}{\rho^2+h^2}\mathcal K,
\]
adding the two integration-by-parts identities yields
\begin{equation}\label{eq:midpoint-identity}
 \int_{I\times J}(2\mathcal K-\mathcal K_I-\mathcal K_J)
 =(n-2s)\int_{I\times J}
 \frac{h^2-d_0h}{\rho^2+h^2}\mathcal K.
\end{equation}
The pushforward of $\dd s\,\dd s'$ under $h=s-s'$ has density
\[
 \varrho(h)=|I\cap(J+h)|=F(|h-d_0|),
\]
where $F$ is nonincreasing. If $d_0\ge0$ and $h>0$, then $|h-d_0|\le|h+d_0|$, hence $\varrho(h)\ge\varrho(-h)$; the inequality reverses for $d_0\le0$. Pairing positive and negative $h$ against the odd function
\[
 h\longmapsto \frac{h}{\rho^2+h^2}\mathcal K(h)
\]
gives
\[
 d_0\int_{I\times J}\frac{h}{\rho^2+h^2}\mathcal K\ge0.
\]
Thus \eqref{eq:midpoint-identity} is at most
\[
 (n-2s)\int_{I\times J}\mathcal D\mathcal K.
\]
Multiply by $c=\min(c_I,c_J)$. Each remaining individual endpoint defect is bounded by $\int_{I\times J}\mathcal K$ because its endpoint average is nonnegative. The leftover coefficients sum to $|c_I-c_J|$, which gives \eqref{eq:interval-ineq}.
\end{proof}

\begin{remark}\label{rem:line-divergence}
When $2s>n-1$, the one-dimensional kernel $\mathcal K(h)$ is not integrable at infinity. The nonnegative formulation in \cref{lem:interval} is therefore essential: no subtraction is performed on an unbounded interval with positive weight.
\end{remark}

\section{Pressure from source layers}\label{sec:pressure}

We next estimate the directional pressure term that remains in the localized virial identity.

\begin{proposition}[Localized source-layer pressure]\label{prop:pressure}
For every $\eps>0$, there is $C_\eps$ such that
\begin{equation}\label{eq:pressure}
 aW_f+bW_g
 \le
 (n-2s)\iint w(x)D(x,y)\,\dd\pi(x,y)
 +\eps E+C_\eps.
\end{equation}
\end{proposition}

\subsection{Angular localization}

Let $\E_e$ and $\Pp_e$ denote expectation and probability with respect to normalized surface measure on $\Sph^{n-1}$. Fix $0<\kappa<1$ and $0<\delta<1/4$. Retain
\[
 \ell(x)=\delta d(x)^m.
\]
Choose a nonnegative smooth $\psi$ on $[0,\infty)$, equal to one near zero, supported in $[0,1]$, and bounded by one. For $0<r\le1$ and $\theta\in\Sph^{n-1}$, set
\[
 Z(r)=\E_e\psi(|e-\theta|/r).
\]
Rotational invariance makes $Z$ independent of $\theta$, and standard spherical coordinates give
\begin{equation}\label{eq:Z-est}
 cr^{n-1}\le Z(r)\le Cr^{n-1},
 \qquad
 |Z'(r)|\le Cr^{n-2}.
\end{equation}
Define
\begin{equation}\label{eq:he}
 h_e(x)=
 \begin{cases}
 \displaystyle
 w(x)\frac{\psi(|e-\omega_x|/(\kappa d(x)))}{Z(\kappa d(x))},&1<|x|<2,\\[2mm]
 0,&\text{otherwise}.
 \end{cases}
\end{equation}
Then
\begin{equation}\label{eq:he-normalization}
 \E_e h_e(x)=w(x),
\end{equation}
and
\begin{equation}\label{eq:he-derivative}
 0\le h_e(x)\le C_\kappa d(x)^{N-n},
 \qquad
 |\nabla h_e(x)|\le C_\kappa d(x)^{N-n-1}.
\end{equation}

\begin{lemma}[Angular errors]\label{lem:angular-errors}
One has
\begin{equation}\label{eq:he-mismatch}
 \E_e\iint|h_e(x)-h_e(y)|\,\dd\pi(x,y)
 \le C_\kappa\delta E+C_{\kappa,\delta},
\end{equation}
and, with
\[
 D_e(x,y)=\frac{((x-y)\cdot e)^2}{|x-y|^2},
\]
\begin{equation}\label{eq:direction-error}
 \E_e\iint h_e(x)D_e(x,y)\,\dd\pi(x,y)
 \le
 \iint w(x)D(x,y)\,\dd\pi(x,y)+C\kappa E.
\end{equation}
\end{lemma}

\begin{proof}
On a close pair, integrate \eqref{eq:he-derivative} along the segment to obtain
\[
 |h_e(x)-h_e(y)|
 \le C_\kappa\delta d(z)^{N-n-1+m}
 =C_\kappa\delta d(z)^{N+2}
 \le C_\kappa\delta d(z)^N.
\]
The interaction marginals give a contribution $C_\kappa\delta E$. On far pairs, \eqref{eq:he-derivative} and \eqref{eq:tail} give, for the $x$-endpoint contribution,
\[
 C_\kappa\int_{B_2}f(x)d(x)^{N-n}\ell(x)^{2s-n}\,\dd x,
\]
whose $d$-exponent is
\[
 N-n-m(n-2s)=n+2+2s(n+3)>0.
\]
Thus the far contribution is bounded by $C_{\kappa,\delta}$, proving \eqref{eq:he-mismatch}.

On the support of $h_e(x)$,
\[
 |D_e(x,y)-D(x,y)|\le2|e-\omega_x|\le2\kappa d(x).
\]
Average, use \eqref{eq:he-normalization} and $dw\le C\eta$, and apply the marginal identity to obtain \eqref{eq:direction-error}.
\end{proof}

\subsection{Measurable layer intervals and their loss}

For $0<t<f(x)$, let $I$ be the connected component containing $x$ of the one-dimensional superlevel set $\{f>t\}$ on the line $x+\R e$. Assign
\begin{equation}\label{eq:chi-I}
 \chi_I=
 \begin{cases}
 \inf_I h_e,& I\text{ bounded and }|I|\le2\delta\inf_I d,\\
 0,&\text{otherwise}.
 \end{cases}
\end{equation}
Set $\chi_f(e,x,t)=\chi_I$. Define $\chi_g(e,y,t')$ analogously. Then
\[
 0\le\chi_f(e,x,t)\le h_e(x),
 \qquad
 0\le\chi_g(e,y,t')\le h_e(y).
\]
The exit distances of these components, and hence their endpoints, depend measurably on the parameters. Indeed, the condition that a fixed compact line segment remain inside an open superlevel set is open, while the infimum in \eqref{eq:chi-I} may be computed over rational offsets. This allows us to use Tonelli's theorem without further comment below.

Define the nonnegative losses
\begin{align}
 L_f(e)&=\int u(x)\int_0^{f(x)}(h_e(x)-\chi_f(e,x,t))\,\dd t\,\dd x,\label{eq:Lf}\\
 L_g(e)&=\int v(y)\int_0^{g(y)}(h_e(y)-\chi_g(e,y,t'))\,\dd t'\,\dd y.\label{eq:Lg}
\end{align}

\begin{lemma}[Averaged layer loss]\label{lem:layer-loss}
For every $P\ge1$,
\begin{equation}\label{eq:layer-loss}
 \E_e(L_f(e)+L_g(e))
 \le
 \bigl(C_\kappa\delta+C_{\kappa,\delta}P^{-1}\bigr)E+CP.
\end{equation}
\end{lemma}

\begin{proof}
It is enough to treat the $f,u$ term. Fix $x$ with $1<|x|<2$, write $d=d(x)$ and $\ell=\delta d^m$. If a component $I$ has $|I|\le\ell$, then
\[
 \inf_I d\ge d-\ell\ge(1-\delta)d,
 \qquad
 |I|\le\delta d\le2\delta\inf_I d,
\]
so it is retained. The derivative estimate \eqref{eq:he-derivative} gives an entire weight loss at most $C_\kappa\delta d^N$.

If $|I|>\ell$, at least one of its two portions from $x$ contains a segment of length $\ell/2$. Let $F$ be the directions whose forward segment of that length remains in $\{f>t\}$. Polar integration centered at $x$ gives
\[
 \int_{B_{\ell/2}(x)}f
 \ge c\,t\ell^n\Pp_e(F).
\]
The long-interval event is contained in $F\cup(-F)$, hence, using \eqref{eq:L1-local},
\begin{equation}\label{eq:long-prob}
 \Pp_e\{|I|>\ell\}\le\frac{C}{t\ell^n}.
\end{equation}
For high levels $t\ge Pd^{-k}$, the expected long-interval loss is bounded by the supremum of $h_e(x)$ times \eqref{eq:long-prob}, hence by
\[
 C_\kappa d^{N-n}\,C\delta^{-n}P^{-1}d^{k-mn}
 =C_{\kappa,\delta}P^{-1}d^N,
\]
because $k=n(m+1)$. Integrating the short- and high-long losses in $t$ and then against $u(x)\,\dd x$ gives
\[
 \bigl(C_\kappa\delta+C_{\kappa,\delta}P^{-1}\bigr)E_f.
\]

For the remaining levels, use \eqref{eq:he-normalization}:
\[
 \int_0^{\min\{f(x),Pd^{-k}\}}
 \E_e(h_e(x)-\chi_f(e,x,t))\,\dd t
 \le Pd^{-k}w(x)
 \le CPd^{N-1-k}\le CP,
\]
since $N-1-k=n+1$. Integrating against $u$ and using \eqref{eq:L1-local} gives $CP$. The proof for $g,v$ is identical.
\end{proof}

\subsection{Endpoint recovery and proof of the pressure estimate}

Suppose $\chi_I>0$ for an $f$-layer at level $t$. Then $I$ is bounded, and continuity plus maximality gives
\[
 f(x_-)=f(x_+)=t.
\]
Since $f=u^q$,
\begin{equation}\label{eq:endpoint-exact}
 u(x_-)=u(x_+)=t^{1/q}.
\end{equation}
Write $u_I=(u(x_-)+u(x_+))/2=t^{1/q}$. Define $v_J=t'^{1/p}$ analogously for a retained $g$-component. If a component weight is zero, every product involving its endpoint average is defined to be zero.

Fix a direction $e$ and levels $t,t'>0$. On two distinct parallel lines, apply \cref{lem:interval} to every pair of layer components $I,J$ with $c_I=\chi_I$, $c_J=\chi_J$. Since $\min(\chi_I,\chi_J)\le h_e(x)$,
\begin{align}
 \int_{I\times J}(\chi_I+\chi_J)\mathcal K
 \le
 \int_{I\times J}\Bigl[
 \chi_I\mathcal K_I+\chi_J\mathcal K_J
 +\bigl((n-2s)h_e(x)D_e(x,y)+|\chi_I-\chi_J|\bigr)\mathcal K
 \Bigr].\label{eq:layer-interval}
\end{align}
All terms are nonnegative extended integrals. Sum over components and integrate in the transverse variables and in $t,t'$. Because $n\ge2$, the coincident-line set has zero measure in $e^\perp\times e^\perp$ and can be omitted. Tonelli gives
\begin{align*}
 U_f(e)&=\int\int_0^{f(x)}\chi_f(e,x,t)u(x)\,\dd t\,\dd x,\\
 U_g(e)&=\int\int_0^{g(y)}\chi_g(e,y,t')v(y)\,\dd t'\,\dd y,
\end{align*}
for the two left-hand terms, while the endpoint terms become
\begin{align*}
 V_f(e)&=\int\int_0^{f(x)}\chi_f(e,x,t)u_I\,\dd t\,\dd x,\\
 V_g(e)&=\int\int_0^{g(y)}\chi_g(e,y,t')v_J\,\dd t'\,\dd y.
\end{align*}
This follows because integrating the endpoint kernels against all $g$-layers reconstructs $K_s*g=u$; the argument for $v$ is identical.

Set
\[
 H_f(e)=\int h_e fu,
 \qquad
 H_g(e)=\int h_e gv.
\]
Each of these aggregated quantities is finite before any subtraction is made. Indeed, $U_f\le H_f$, $U_g\le H_g$, and by \eqref{eq:endpoint-exact}, $0\le V_f\le U_f$ and $0\le V_g\le U_g$. Also $D_e\le1$. Finally,
\begin{align*}
 |\chi_f(e,x,t)-\chi_g(e,y,t')|
 &\le |h_e(x)-h_e(y)|\\
 &\quad +(h_e(x)-\chi_f(e,x,t))+(h_e(y)-\chi_g(e,y,t')).
\end{align*}
Accordingly, the finite quantity
\begin{align*}
 S_e&=\int\int_0^{f(x)}\chi_f(e,x,t)(u(x)-u_I)\,\dd t\,\dd x\\
 &\quad+\int\int_0^{g(y)}\chi_g(e,y,t')(v(y)-v_J)\,\dd t'\,\dd y
\end{align*}
satisfies
\begin{equation}\label{eq:Se}
 S_e\le
 (n-2s)\iint h_e(x)D_e(x,y)\,\dd\pi
 +\iint|h_e(x)-h_e(y)|\,\dd\pi
 +L_f(e)+L_g(e).
\end{equation}

The exact layer identities are
\begin{equation}\label{eq:layer-id-f}
 \int_0^{f(x)}\bigl(u(x)-t^{1/q}\bigr)\,\dd t
 =\frac{f(x)u(x)}{q+1}=af(x)u(x),
\end{equation}
and
\begin{equation}\label{eq:layer-id-g}
 \int_0^{g(y)}\bigl(v(y)-t^{1/p}\bigr)\,\dd t
 =\frac{g(y)v(y)}{p+1}=bg(y)v(y).
\end{equation}
By \eqref{eq:endpoint-exact}, for every layer,
\[
 h_e(x)(u-t^{1/q})-\chi_f(e,x,t)(u-u_I)
 =(h_e(x)-\chi_f(e,x,t))(u-t^{1/q})
 \le(h_e(x)-\chi_f(e,x,t))u.
\]
We therefore obtain
\begin{equation}\label{eq:recover-pressure}
 aH_f(e)+bH_g(e)\le S_e+L_f(e)+L_g(e).
\end{equation}
Combining \eqref{eq:Se} with \eqref{eq:recover-pressure}, averaging over $e$, and using \cref{lem:angular-errors,lem:layer-loss} together with $\E_eH_f=W_f$ and $\E_eH_g=W_g$, we obtain
\begin{align}
 aW_f+bW_g
 &\le (n-2s)\iint w(x)D(x,y)\,\dd\pi\nonumber\\
 &\quad+\bigl(C\kappa+C_\kappa\delta+C_{\kappa,\delta}P^{-1}\bigr)E
 +C_{\kappa,\delta,P}.\label{eq:pressure-parameters}
\end{align}
Choosing the parameters in the order $\kappa\downarrow0$, then $\delta\downarrow0$, and finally $P\uparrow\infty$ proves \cref{prop:pressure}.

\section{Energy bound and scaling contradiction}\label{sec:closure}

We can now close the argument. Define
\begin{equation}\label{eq:T-Q}
 T=\iint w(x)D(x,y)\,\dd\pi(x,y),
 \qquad
 Q=\iint(\Phi(x,y)-\eta(x)+w(x)D(x,y))\,\dd\pi(x,y).
\end{equation}
By \cref{lem:localization}, for every $\eps>0$,
\begin{equation}\label{eq:Q-bound}
 |Q|+|E_f-E_g|\le\eps E+C_\eps.
\end{equation}
Since
\[
 \iint\Phi\,\dd\pi=E_f-T+Q,
\]
the localized virial identity \eqref{eq:localized-virial} becomes
\begin{equation}\label{eq:virial-rearranged}
 c_1E_f+c_2E_g
 =aW_f+bW_g-(n-2s)T+(n-2s)Q,
\end{equation}
where
\begin{equation}\label{eq:c1c2}
 c_1=an-(n-2s),
 \qquad
 c_2=bn.
\end{equation}
The pressure estimate \eqref{eq:pressure} and \eqref{eq:Q-bound} imply
\begin{equation}\label{eq:weighted-energy-upper}
 c_1E_f+c_2E_g\le C\eps E+C_\eps.
\end{equation}
Set
\begin{equation}\label{eq:gamma}
 \gamma=c_1+c_2
 =\frac{n}{q+1}+\frac{n}{p+1}-(n-2s).
\end{equation}
The subcritical assumption is exactly $\gamma>0$. Moreover,
\[
 c_1E_f+c_2E_g
 =\frac\gamma2E+\frac{c_1-c_2}{2}(E_f-E_g).
\]
Thus \eqref{eq:Q-bound} and \eqref{eq:weighted-energy-upper} give
\[
 \frac\gamma2E\le C\eps E+C_\eps.
\]
Choosing $\eps$ so that $C\eps\le\gamma/4$ yields the uniform local energy bound
\begin{equation}\label{eq:universal-energy}
 \boxed{E\le C_{n,s,p,q}.}
\end{equation}

The last step is to compare this uniform bound with the scaling of the system. From \eqref{eq:alpha-beta} and \eqref{eq:a-b},
\begin{equation}\label{eq:gamma-scaling}
 \gamma
 =\left(1-\frac1{p+1}-\frac1{q+1}\right)
 (\alpha+\beta+2s-n).
\end{equation}
The first factor is positive because $pq>1$. Hence subcriticality gives
\begin{equation}\label{eq:delta-positive}
 \delta_*:=\alpha+\beta+2s-n>0.
\end{equation}
Apply \eqref{eq:universal-energy} to the scaled solution $(u_R,v_R)$. Since $\eta=1$ on $B_1$,
\begin{align*}
 E[u_R,v_R]
 &\ge\int_{B_1}\bigl(u_R^{q+1}+v_R^{p+1}\bigr)\,\dd x\\
 &=R^{\delta_*}\int_{B_R}\bigl(u^{q+1}+v^{p+1}\bigr)\,\dd x.
\end{align*}
For $R\ge1$, the last integral is bounded below by the fixed positive number
\[
 c_0=\int_{B_1}(u^{q+1}+v^{p+1})\,\dd x>0.
\]
Hence
\[
 E[u_R,v_R]\ge c_0R^{\delta_*}\longrightarrow\infty,
\]
which contradicts the solution-independent bound \eqref{eq:universal-energy}. The proof of \cref{thm:main} is complete.

\section{Further remarks}

\begin{remark}[PDE formulation]
The integral system \eqref{eq:system} is the exact Riesz-potential form of
\[
 (-\Delta)^s u=v^p,\qquad (-\Delta)^s v=u^q.
\]
Proposition~\ref{prop:pde-potential} shows that the two formulations agree in the whole-space class used in this paper. More generally, the interval-pressure argument applies to any PDE solution class for which the same potential representation is available without a nonzero $s$-harmonic remainder.
\end{remark}

\begin{remark}[Why the unweighted problem is simpler than the weighted H\'enon case]
For a weighted layer $|x|^B u(x)^q>t$, the endpoint relation is $u=t^{1/q}|x|^{-B/q}$, so the endpoint value changes along the component and creates an additional short-interval error. In the present unweighted problem, every retained endpoint satisfies the exact identity $u_I=t^{1/q}$. This removes the endpoint-weight error from the pressure recovery step \eqref{eq:recover-pressure}. For the local weighted problem, see \cite{FazlyGhoussoub2014,HuangZou2026,OpenAI2026,Phan2012}.
\end{remark}

\begin{remark}[The one-dimensional case]
The argument uses $n\ge2$ at one specific point: after decomposing space into lines parallel to $e$, the coincident-line set is negligible in $e^\perp\times e^\perp$. This fails in dimension one, so the case $n=1$ requires a separate argument.
\end{remark}

\begin{remark}[Possible extensions]
Two natural extensions are weighted systems \cite{DaiQin2021,Hajlaoui2022} and systems with unequal fractional orders. The same geometric mechanism should still be relevant, but the endpoint layer values and the homogeneity coefficients no longer match, so the pressure comparison would need additional bookkeeping.
\end{remark}

\section*{Acknowledgments}
The proof is inspired by the localized potential-kernel virial and source-layer interval-pressure method introduced in the 2026 preprint \emph{The Subcritical H\'enon--Lane--Emden Conjecture} by OpenAI \cite{OpenAI2026}. Our argument isolates the kernel-level part of that method and adapts it from the Newton kernel to the fractional Riesz kernel $|x|^{2s-n}$, $0<s<1$.


\begin{thebibliography}{99}

\bibitem{AdamsHedberg1996}
D. R. Adams and L. I. Hedberg,
\emph{Function Spaces and Potential Theory},
Grundlehren der mathematischen Wissenschaften, vol. 314,
Springer, Berlin, 1996.

\bibitem{BidautVeron1989}
M.-F. Bidaut-V\'eron,
Local and global behavior of solutions of quasilinear equations of Emden--Fowler type,
\emph{Arch. Rational Mech. Anal.} \textbf{107} (1989), 293--324.

\bibitem{BidautVeronGiacomini2010}
M.-F. Bidaut-V\'eron and H. Giacomini,
A new dynamical approach of Emden--Fowler equations and systems,
\emph{Adv. Differential Equations} \textbf{15} (2010), 1033--1082.

\bibitem{CaffarelliSilvestre2007}
L. Caffarelli and L. Silvestre,
An extension problem related to the fractional Laplacian,
\emph{Comm. Partial Differential Equations} \textbf{32} (2007), 1245--1260.

\bibitem{CaristiDAmbrosioMitidieri2008}
G. Caristi, L. D'Ambrosio, and E. Mitidieri,
Representation formulae for solutions to some classes of higher order systems and related Liouville theorems,
\emph{Milan J. Math.} \textbf{76} (2008), 27--67.

\bibitem{ChenFangYang2015}
W. Chen, Y. Fang, and R. Yang,
Liouville theorems involving the fractional Laplacian on a half space,
\emph{Adv. Math.} \textbf{274} (2015), 167--198.

\bibitem{ChenHuyuan2018}
H. Chen,
Liouville theorem for the fractional Lane--Emden equation in an unbounded domain,
\emph{J. Math. Pures Appl.} \textbf{111} (2018), 21--46.

\bibitem{ChenJinLiLim2005}
W. Chen, C. Jin, C. Li, and J. Lim,
Weighted Hardy--Littlewood--Sobolev inequalities and systems of integral equations,
\emph{Discrete Contin. Dyn. Syst. Suppl.} (2005), 164--172.

\bibitem{ChenLi2009}
W. Chen and C. Li,
An integral system and the Lane--Emden conjecture,
\emph{Discrete Contin. Dyn. Syst.} \textbf{24} (2009), 1167--1184.

\bibitem{ChenLiOu2005}
W. Chen, C. Li, and B. Ou,
Classification of solutions for a system of integral equations,
\emph{Comm. Partial Differential Equations} \textbf{30} (2005), 59--65.

\bibitem{ChenLiOu2006}
W. Chen, C. Li, and B. Ou,
Classification of solutions for an integral equation,
\emph{Comm. Pure Appl. Math.} \textbf{59} (2006), 330--343.

\bibitem{ChenSong1998}
Z.-Q. Chen and R. Song,
Estimates on Green functions and Poisson kernels for symmetric stable processes,
\emph{Math. Ann.} \textbf{312} (1998), 465--501.

\bibitem{ChengHuang2019}
Z. Cheng and G. Huang,
A Liouville theorem for the subcritical Lane--Emden system,
\emph{Discrete Contin. Dyn. Syst.} \textbf{39} (2019), 1359--1377.

\bibitem{ChengHuangLi2016}
Z. Cheng, G. Huang, and C. Li,
On the Hardy--Littlewood--Sobolev type systems,
\emph{Commun. Pure Appl. Anal.} \textbf{15} (2016), 2059--2074.

\bibitem{DaiLiuLu2016}
W. Dai, Z. Liu, and G. Lu,
Liouville type theorems for PDE and IE systems involving fractional Laplacian on a half space,
\emph{Potential Anal.} \textbf{46} (2016), 569--588.

\bibitem{DaiQin2021}
W. Dai and G. Qin,
Liouville type theorem for critical order H\'enon--Lane--Emden type equations on a half space and its applications,
\emph{J. Funct. Anal.} \textbf{281} (2021), 109227.

\bibitem{DAmbrosioMitidieri2014}
L. D'Ambrosio and E. Mitidieri,
Hardy--Littlewood--Sobolev systems and related Liouville theorems,
\emph{Discrete Contin. Dyn. Syst. Ser. S} \textbf{7} (2014), 653--671.

\bibitem{DavilaDupaigneWei2017}
J. D\'avila, L. Dupaigne, and J. Wei,
On the fractional Lane--Emden equation,
\emph{Trans. Amer. Math. Soc.} \textbf{369} (2017), 6087--6104.

\bibitem{DiNezza2012}
E. Di Nezza, G. Palatucci, and E. Valdinoci,
Hitchhiker's guide to the fractional Sobolev spaces,
\emph{Bull. Sci. Math.} \textbf{136} (2012), 521--573.

\bibitem{DouZhu2015}
M. Dou and M. Zhu,
Sharp Hardy--Littlewood--Sobolev inequality on the upper half space,
\emph{Int. Math. Res. Not.} \textbf{2015} (2015), 651--687.

\bibitem{DuongNguyen2021}
A. T. Duong and V. H. Nguyen,
Liouville type theorems for some fractional elliptic problems,
\emph{Nonlinear Anal.} \textbf{210} (2021), 112383.

\bibitem{DupaigneFarinaPetitt2023}
L. Dupaigne, A. Farina, and T. Petitt,
Liouville-type theorems for the Lane--Emden equation in the half-space and cones,
\emph{J. Funct. Anal.} \textbf{284} (2023), 109906.

\bibitem{DupaigneSirakovSouplet2022}
L. Dupaigne, B. Sirakov, and P. Souplet,
A Liouville-type theorem for the Lane--Emden equation in a half-space,
\emph{Int. Math. Res. Not.} \textbf{2022} (2022), 9024--9043.

\bibitem{FallWeth2012}
M. M. Fall and T. Weth,
Nonexistence results for a class of fractional elliptic boundary value problems,
\emph{J. Funct. Anal.} \textbf{263} (2012), 2205--2227.

\bibitem{FallWeth2016}
M. M. Fall and T. Weth,
Monotonicity and nonexistence results for some fractional elliptic problems in the half-space,
\emph{Commun. Contemp. Math.} \textbf{18} (2016), 1550012.

\bibitem{FazlyGhoussoub2014}
M. Fazly and N. Ghoussoub,
On the H\'enon--Lane--Emden conjecture,
\emph{Discrete Contin. Dyn. Syst.} \textbf{34} (2014), 2513--2533.

\bibitem{FazlyWei2017}
M. Fazly and J. Wei,
On finite Morse index solutions of higher order fractional Lane--Emden equations,
\emph{Amer. J. Math.} \textbf{139} (2017), 433--460.

\bibitem{GidasSpruck1981a}
B. Gidas and J. Spruck,
Global and local behavior of positive solutions of nonlinear elliptic equations,
\emph{Comm. Pure Appl. Math.} \textbf{34} (1981), 525--598.

\bibitem{GidasSpruck1981b}
B. Gidas and J. Spruck,
A priori bounds for positive solutions of nonlinear elliptic equations,
\emph{Comm. Partial Differential Equations} \textbf{6} (1981), 883--901.

\bibitem{Hajlaoui2022}
H. Hajlaoui,
Liouville type theorems for stable solutions of the weighted fractional Lane--Emden system,
\emph{Discrete Contin. Dyn. Syst.} \textbf{42} (2022), 5665--5681.

\bibitem{Hang2007}
F. Hang,
On the integral systems related to Hardy--Littlewood--Sobolev inequality,
\emph{Math. Res. Lett.} \textbf{14} (2007), 373--383.

\bibitem{HuangZou2026}
L.-H. Huang and W. Zou,
Liouville-type theorems for stable solutions of the H\'enon--Lane--Emden system,
\emph{J. London Math. Soc.} \textbf{113} (2026), e70412.

\bibitem{JarohsWeth2016}
S. Jarohs and T. Weth,
Symmetry via antisymmetric maximum principles in nonlocal problems of variable order,
\emph{Ann. Mat. Pura Appl.} \textbf{195} (2016), 273--291.

\bibitem{KuiLiLiWei2025}
K. Li, M. Li, and J. Wei,
On a new region for the Lane--Emden conjecture in higher dimensions,
\emph{arXiv:2510.06613} (2025).

\bibitem{Landkof1972}
N. S. Landkof,
\emph{Foundations of Modern Potential Theory},
Grundlehren der mathematischen Wissenschaften, vol. 180,
Springer, Berlin, 1972.

\bibitem{Le2021}
P. Le,
Method of scaling spheres for integral and polyharmonic systems,
\emph{J. Differential Equations} \textbf{298} (2021), 132--158.

\bibitem{Lieb1983}
E. H. Lieb,
Sharp constants in the Hardy--Littlewood--Sobolev and related inequalities,
\emph{Ann. of Math. (2)} \textbf{118} (1983), 349--374.

\bibitem{Li2022}
H. Li,
A Liouville theorem for the H\'enon--Lane--Emden system in four and five dimensions,
\emph{Adv. Nonlinear Stud.} \textbf{22} (2022), 517--533.

\bibitem{LiSouplet2025}
Y. Li and P. Souplet,
A Liouville theorem for the Lane--Emden system in the half-space,
\emph{J. Funct. Anal.} \textbf{289} (2025), 111107.

\bibitem{LiZhang2019}
K. Li and Z. Zhang,
Proof of the H\'enon--Lane--Emden conjecture in $\R^3$,
\emph{J. Differential Equations} \textbf{266} (2019), 202--226.

\bibitem{Mitidieri1993}
E. Mitidieri,
A Rellich type identity and applications,
\emph{Comm. Partial Differential Equations} \textbf{18} (1993), 125--151.

\bibitem{Mitidieri1996}
E. Mitidieri,
Nonexistence of positive solutions of semilinear elliptic systems in $\R^N$,
\emph{Differential Integral Equations} \textbf{9} (1996), 465--479.

\bibitem{MitidieriPohozaev2001}
E. Mitidieri and S. I. Pohozaev,
A priori estimates and the absence of solutions of nonlinear partial differential equations and inequalities,
\emph{Tr. Mat. Inst. Steklova} \textbf{234} (2001), 1--384.

\bibitem{OpenAI2026}
OpenAI,
\emph{The Subcritical H\'enon--Lane--Emden Conjecture},
preprint, September 24, 2026.
\url{https://github.com/openai/math/blob/main/preprints/The-Subcritical-Henon-Lane-Emden-Conjecture-September-24-2026/paper.pdf}.

\bibitem{Phan2012}
Q. H. Phan,
Liouville-type theorems and bounds of solutions for Hardy--H\'enon elliptic systems,
\emph{Adv. Differential Equations} \textbf{17} (2012), 605--634.

\bibitem{PolacikQuittnerSouplet2007}
P. Pol\'a\v{c}ik, P. Quittner, and P. Souplet,
Singularity and decay estimates in superlinear problems via Liouville-type theorems. Part I: Elliptic equations and systems,
\emph{Duke Math. J.} \textbf{139} (2007), 555--579.

\bibitem{QuaasXia2014}
A. Quaas and A. Xia,
Liouville type theorems for nonlinear elliptic equations and systems involving fractional Laplacian in the half space,
\emph{Calc. Var. Partial Differential Equations} \textbf{52} (2014), 641--659.

\bibitem{QuaasXia2016}
A. Quaas and A. Xia,
A Liouville type theorem for Lane--Emden systems involving the fractional Laplacian,
\emph{Nonlinearity} \textbf{29} (2016), 2279--2297.

\bibitem{RosOton2016}
X. Ros-Oton,
Nonlocal elliptic equations in bounded domains: a survey,
\emph{Publ. Mat.} \textbf{60} (2016), 3--26.

\bibitem{RosOtonSerra2014a}
X. Ros-Oton and J. Serra,
The Pohozaev identity for the fractional Laplacian,
\emph{Arch. Ration. Mech. Anal.} \textbf{213} (2014), 587--628.

\bibitem{RosOtonSerra2014b}
X. Ros-Oton and J. Serra,
The Dirichlet problem for the fractional Laplacian: regularity up to the boundary,
\emph{J. Math. Pures Appl.} \textbf{101} (2014), 275--302.

\bibitem{RosOtonSerra2015}
X. Ros-Oton and J. Serra,
Nonexistence results for nonlocal equations with critical and supercritical nonlinearities,
\emph{Comm. Partial Differential Equations} \textbf{40} (2015), 115--133.

\bibitem{SerrinZou1996}
J. Serrin and H. Zou,
Non-existence of positive solutions of Lane--Emden systems,
\emph{Differential Integral Equations} \textbf{9} (1996), 635--653.

\bibitem{Silvestre2007}
L. Silvestre,
Regularity of the obstacle problem for a fractional power of the Laplace operator,
\emph{Comm. Pure Appl. Math.} \textbf{60} (2007), 67--112.

\bibitem{Souplet2009}
P. Souplet,
The proof of the Lane--Emden conjecture in four space dimensions,
\emph{Adv. Math.} \textbf{221} (2009), 1409--1427.

\bibitem{SteinWeiss1958}
E. M. Stein and G. Weiss,
Fractional integrals on $n$-dimensional Euclidean space,
\emph{J. Math. Mech.} \textbf{7} (1958), 503--514.

\bibitem{Villavert2015}
J. Villavert,
Sharp existence criteria for positive solutions of Hardy--Sobolev type systems,
\emph{Commun. Pure Appl. Anal.} \textbf{14} (2015), 493--515.

\bibitem{ZhuoChenCuiYuan2014}
R. Zhuo, W. Chen, X. Cui, and Z. Yuan,
A Liouville theorem for the fractional Laplacian,
\emph{arXiv:1401.7402} (2014).

\end{thebibliography}
\end{document}